\documentclass[a4paper,twoside]{amsart}
\usepackage[utf8]{inputenc}
\usepackage[hidelinks]{hyperref}
\usepackage{graphicx}
\usepackage{enumitem}

\theoremstyle{definition}

\newtheorem*{remark*}{Remark}
\newtheorem*{claim*}{Claim}
\theoremstyle{remark}

\theoremstyle{plain}
\newtheorem{theorem}{Theorem}[section]

\newtheorem{prop}[theorem]{Proposition}
\newtheorem{observation}[theorem]{Observation}

\newcommand{\overbar}[1]{\mkern 1.5mu\overline{\mkern-1.5mu#1\mkern-1.5mu}\mkern 1.5mu}
\def\dom{\mathop{\mathrm{Dom}}\nolimits}
\def\posreals{\mathbb R^{>0}}

\def\str#1{\mathbf {#1}}

\begin{document}
\bibliographystyle{alpha}

\title[]{A combinatorial proof of the extension property for partial isometries}

\authors{
\author[J. Hubi\v cka]{Jan Hubi\v cka}
\address{Department of Applied Mathematics (KAM)\\ Charles University\\ Prague, Czech Republic}
\email{hubicka@kam.mff.cuni.cz}
\author[M. Kone\v cn\'y]{Mat\v ej Kone\v cn\'y}
\address{Charles University\\ Prague, Czech Republic}
\email{matejkon@gmail.com}
\author[J. Ne\v set\v ril]{Jaroslav Ne\v set\v ril}
\address{Computer Science Institute of Charles University (IUUK)\\ Charles University\\ Prague, Czech Republic}
\email{nesetril@iuuk.mff.cuni.cz}
\thanks{Jan Hubi\v cka and Mat\v ej Kone\v cn\'y are supported by project 18-13685Y of the Czech Science Foundation (GA\v CR)}
}

\begin{abstract}
We present a short and self-contained proof of the extension property for partial
isometries of the class of all finite metric spaces.
\end{abstract}
\maketitle
\section{Introduction}
A class of metric spaces $\mathcal C$ has the \emph{extension property for
partial isometries} if for every $\str{A}\in \mathcal C$ there exists
$\str{B}\in \mathcal C$ containing $\str{A}$ as a subspace with the property
that every isometry of two subspaces of $\str{A}$ extends to an isometry of
$\str{B}$. (By isometry we mean a bijective distance-preserving function.)
In this note we give a self-contained combinatorial proof of the following theorem:
\begin{theorem}[Solecki~\cite{solecki2005}, Vershik~\cite{vershik2008}]
\label{mainresult}
The class of all finite metric spaces has the extension property for partial
isometries.
\end{theorem}
This result is important from the point of view of combinatorics, model theory
as well as topological dynamics.  It has several proofs~\cite{solecki2005,Pestov2008,rosendal2011}, \cite[Theorem 8.3]{sabok2017automatic}
which are based on deep group-theoretic results (the M.~Hall theorem~\cite{hall1949}, the Herwig--Lascar
theorem~\cite{herwig2000,otto2017,Siniora}, the {R}ibes--{Z}alesski{\u\i} thorem \cite{Ribes1993} or Mackey's
construction~\cite{mackey1966}).  Vershik announced an elementary proof \cite{vershik2008}
which remains unpublished and differs from the approach presented
here \cite{Vershikprivate}.

Our construction is elementary. We follow a general strategy analogous to the corresponding
results about the existence of Ramsey expansions of the class of finite metric
spaces developed in series of
papers~\cite{Nevsetvril1977,Nevsetvril2007,Hubicka2016}.  
We proceed in two steps.

First, given a metric space $\str A$, we find an edge-labelled graph $\str B_0$ which extends all
partial isometries of $\str A$, but does not define all distances between vertices and may not have
a completion to metric space (for example, it may contain non-metric triangles).
This step is analogous to the easy combinatorial proof of Hrushovski's theorem by
Herwig and Lascar~\cite{herwig2000}.

In the second step we further expand and ``sparsify'' $\str B_0$ in order to remove all
obstacles which prevent us from being able to define the missing distances and get a metric space. Once all such
obstacles are eliminated, we can complete the edge-labelled graph to a metric space $\str B$ by assigning
every pair of vertices a distance corresponding to the shortest path connecting
them.  This part is inspired by a clique-faithful EPPA construction of Hodkinson and
Otto~\cite{hodkinson2003} (see also Hodkinson's exposition~\cite{hodkinson}).

Similarly to the Ramsey constructions which were developed to work under rather
general structural conditions~\cite{Hubicka2016}, our technique generalises
further to classes described by forbidden homomorphisms as well as to the
classes with algebraic closures (in the sense of~\cite{Evans3}) and antipodal
metric spaces (as shown in~\cite{eppatwographs}). These strengthenings are
going to appear elsewhere.

\section{Notation and preliminaries}
Given a set of labels $L$, an \emph{$L$-edge-labelled graph} is an (undirected) graph where
every edge has a unique \emph{label} $\ell \in L$.  In our proof we use
``partial'' metric spaces (where some distances are not known) and thus we
will consider metric spaces as a special case of $\posreals$-edge-labelled graphs
where $\posreals$ is the set of positive reals: an $\posreals$-edge-labelled graph is
then a \emph{metric space} if it is complete (that is, every pair of vertices is
connected by an edge) and for every triple of distinct vertices $x,y,z$ the
labels of edges $\{x,y\}$, $\{y,z\}$ and $\{x,y\}$ satisfy the triangle inequality.

While we need to work with edge-labelled graphs to represent intermediate objects
in our construction, we find it useful to adopt standard terminology of metric
spaces.  If vertices $x$ and $y$ of an edge-labelled graph $\str{A}$ form an edge
with label $\ell$, we will also say that the edge $\{x,y\}$ has \emph{length}
$\ell$, or write $d_\str{A}(x,y)=d_\str{A}(y,x)=\ell$ and say that $\ell$ is the
\emph{distance} between $x$ and $y$.

We will use bold letters such as $\str{A},\str{B},\str{C},\ldots$ to denote
edge-labelled graphs and the corresponding normal letters ($A,B,C,\ldots$)
to denote the corresponding vertex sets.

\medskip

Given two $L$-edge-labelled graphs $\str{A}$ and $\str{B}$, a function $\varphi\colon A\to
B$ is a \emph{homomorphism} if for every pair of vertices $x,y\in A$ which forms an edge with label
$\ell$ in $\str{A}$ it holds that $\varphi(x),\varphi(y)$ is an edge with label $\ell$ in $\str{B}$.
If $\varphi$ is injective, it is a \emph{monomorphism}. A monomorphism
 where for every $x,y\in A$ it holds that
$x,y$ form an edge with label $\ell$ if and only if $\varphi(x),\varphi(y)$
form an edge with the same label $\ell$ is called \emph{embedding}. If $A\subseteq B$ and the inclusion map is a monomorphism, we say that $\str{A}$ is a \emph{subgraph} of $\str{B}$. A subgraph is \emph{induced} if the inclusion map is an embedding.  A
bijective embedding is an \emph{isomorphism} and an isomorphism
$\str{A}\to\str{A}$ is an \emph{automorphism}.  A \emph{partial automorphism}
of $\str{A}$ is any isomorphism of two induced subgraphs of $\str{A}$.
In the context of metric spaces we sometimes say \emph{isometry} instead of isomorphism.

A \emph{walk} in an edge-labelled graph $\str{A}$ connecting $x_1\in A$ and $x_n\in A$ is any sequence of vertices $x_1,x_2,\ldots, x_n$ such that for every $1\leq i<n$ there
is edge connecting $x_i$ and $x_{i+1}$. The \emph{length of this
walk} is $\sum_{1\leq i<n} d_\str{A}(x_i,x_{i+1})$. A \emph{path} is a walk which contains no repeated
vertices. If there is a path $x_1, \ldots, x_n$ with $n\geq 3$ and there is an edge connecting
$x_1$ and $x_n$ then $x_1,x_2,\ldots, x_n$ is a \emph{cycle}.
A cycle is \emph{non-metric} if it contains a (unique)
edge with label $\ell$ which is greater than sum of labels of all the remaining edges.
We will call this edge the \emph{long edge} of the non-metric cycle.
An $\posreals$-edge-labelled graph $\str{A}$ is \emph{connected} if for every $x,y\in A$ there
exists a path connecting $x$ and $y$.

\medskip

Given a connected $\posreals$-edge-labelled graph $\str{G}$, its \emph{shortest path
completion} is the complete $\posreals$-edge-labelled graph $\overbar{\str{G}}$ on the same vertex set as $\str{G}$
such that the label of $x,y$ in $\overbar{\str{G}}$ is the minimal length of a path connecting $x$ and $y$ in $\str{G}$. We will need the following fact about the shortest path completion.
\begin{observation}
\label{obs:shortest}
For every connected $\posreals$-edge-labelled graph $\str{G}$, its shortest path completion $\overbar{\str{G}}$
is a metric space. $\str{G}$ is a (not necessarily induced) subgraph of $\overbar{\str{G}}$ if and only
if it contains no induced non-metric cycles (that is, no induced subgraphs isomorphic to a non-metric cycle).
Moreover, every automorphism of $\str{G}$ is also an automorphism of $\overbar{\str{G}}$.
\end{observation}
\begin{proof}
For any triple of vertices $x,y,z\in \overbar{\str{G}}$ there are, by definition, paths
$x=x_1,\allowbreak x_2,\allowbreak \ldots, x_n=y$ and $y=x_n,x_{n+1},\ldots, x_m=z$ in $\str{G}$ witnessing the
distances $d_{\overbar{\str{G}}}(x,y)$ and $d_{\overbar{\str{G}}}(y,z)$ respectively.
It follows that $x_1,x_2,\ldots, x_m$ is a walk in $\str{G}$ containing a path connecting $x$ and $z$ of length no greater than $d_{\overbar{\str{G}}}(x,y)+d_{\overbar{\str{G}}}(y,z)$. We thus conclude that $d_{\overbar{\str{G}}}(x,z)\leq d_{\overbar{\str{G}}}(x,y),d_{\overbar{\str{G}}}(y,z)$, that is, the triangle inequality holds, and thus $\overbar{\str{G}}$ indeed is a metric space.

If $\str{G}$ contains a non-metric cycle with the longest edge
between $x$ and $y$, it is easy to see that distance of $x$, $y$ in $\overbar{\str{G}}$
is strictly smaller than the distance of $x$ and $y$ in $\str{G}$. Therefore $\str G$ is not a subgraph of $\overbar{\str G}$.

Next we show that if $\str{G}$ contains no induced non-metric cycles then
it is a subgraph of $\overbar{\str{G}}$. Assume, to the
contrary, that there is a pair of vertices $x,y$ connected by an edge in
$\str{G}$ where the labels differs.  Because $x,y$ is also a path connecting $x$ and $y$ in $\str G$, we know that
the label of $x,y$ in $\str{G}$ is greater than the length of shortest path
connecting $x,y$, hence they together form a
non-metric cycle. This cycle is not necessarily induced but adding an edge to a
non-metric cycle splits it to two cycles where at least one  is necessarily also
non-metric.

\medskip

Finally, to verify that the shortest path completion preserves all
automorphisms observe that every distance in $\overbar{\str{G}}$ corresponds to
a path in $\str{G}$ (and to a lack of any shorter path) and paths are preserved
by every automorphism of $\str{G}$.
\end{proof}

\section{Extending partial automorphisms of $\posreals$-edge-labelled graphs}
\begin{prop}
\label{prop:eppa1}
For every finite $\posreals$-edge-labelled graph $\str{A}$ there exists a finite 
$\posreals$-edge-labelled graph $\str{B}$ containing $\str{A}$ as an induced subgraph such
that every partial automorphism of $\str{A}$ extends to an automorphism of
$\str{B}$.
\end{prop}
What follows is a variant of the easy proof of the extension property for partial automorphisms for graphs in~\cite{herwig2000}.
\begin{proof}
Fix $\str{A}$ and let $S=\{s_1,s_2,\ldots, s_n\}\subseteq \posreals$ be the finite subset of $\posreals$ consisting of
all distances in $\str{A}$ (the \emph{spectrum} of $\str{A}$). First we assign every vertex $x\in A$ the set $\psi(x)$
such that for some fixed $k$ the following is satisfied:
\begin{enumerate}
 \item\label{p1} For every $x\neq y\in A$ such that $d_\str{A}(x,y)=s_j$ and integer $i$ it holds that $(\{x,y\},i)\in \psi(x)$ iff $1\leq i\leq j$.
 \item\label{p2} For every $x\in A$ it holds that $|\psi(x)|=k$.
 \item\label{p3} For every $x\neq y\in A$ it holds that $\psi(x)\cap \psi(y) = \{(\{x,y\},i):1\leq i \leq j\}$, where $d_\str{A}(x,y)=s_j$.
\end{enumerate}
Such a function $\psi$ is easy to build. Assign elements to sets to satisfy
(\ref{p1}) and then extend the sets by arbitrary new elements (for example, natural numbers) to satisfy
(\ref{p2}) where every new element belongs to precisely one set so that (\ref{p3}) holds.

Put $$U=\bigcup_{x\in A} \psi(x)$$
to be the universe of our representation.  We construct $\str{B}$ as follows.
\begin{itemize}
 \item The vertex set $B$ of $\str{B}$ consists of all subsets of $U$ of size $k$ (we will denote them by upper case letters $X$ and $Y$).
 \item A pair of vertices $X,Y\in B$ is connected by an edge of length $s_i$ if and only if $X\neq Y$ and $|X\cap Y|=i$. Otherwise $X,Y$ is a non-edge.
\end{itemize}

It is easy to verify that the structure $\str{A}'$ induced by $\str{B}$ on $\{\psi(x);x\in A\}$ is isomorphic to $\str{A}$, that is, $\psi$ is an embedding of $\str A$ into $\str B$. We claim that every partial automorphism of $\str{A}'$ extends to an automorphism of $\str{B}$.
Fix such a partial automorphism $\varphi'$ of $\str{A}'$. By $\varphi$ we denote the partial automorphism induced by $\varphi'$ on $\str{A}$, i.e.
$\varphi=\psi^{-1} \circ \varphi'\circ \psi$.
Note that every permutation of $U$ gives rise to an automorphism of $\str{B}$. We are going to construct an automorphism $\hat\varphi$ of $\str B$ which extends $\varphi'$ by finding the right permutation $\pi$ by the following procedure:
\begin{enumerate}
 \item Start with the partial permutation $\pi$ mapping $(\{x,y\},i)\mapsto (\{\varphi(x),\varphi(y)\},i)$ for every $x\neq y\in \dom(\varphi)$ and $1\leq i\leq j$ where $d_\str{A}(x,y)=s_j$.
 \item Consider every choice of $x \in \dom(\varphi)$. Let $e$ be element of $\psi(x)$  such that $\pi(e)$ is not defined and put $\pi(e)$ to be any element of $\psi(\varphi(x))$ which is not in the image of $\pi$ yet. This is always possible because all the sets have same size and are disjoint except for elements we already assigned maps to.
 \item The partial permutation $\pi$ can then be extended to a full permutation in an arbitrary way.
\end{enumerate}
It is easy to see that $\pi$ induces an automorphism $\hat\varphi$ of $\str{B}$ and that this automorphism extends $\varphi'$.
\end{proof}
\section{Proof of the main result}
Now we are ready to prove Theorem~\ref{mainresult}.
Similarly as in the proof of Hodkinson--Otto~\cite{hodkinson2003}, we use Proposition~\ref{prop:eppa1} to obtain an $\posreals$-edge-labelled graph $\str{B}$. We then consider all ``bad'' substructures of $\str{B}$ (namely the non-metric cycles)
and eliminate each one independently while preserving all necessary symmetries and a projection (in fact, a homomorphism) to the original structure.  The resulting structure is then a product of all these constructions (however, we will define it explicitly). The extension property for partial automorphisms then follows from the fact that automorphisms of $\str{B}$ are mapping bad substructures to their isomorphic copies and we repaired both of them in the same way.

To simplify the construction, we proceed by induction on the size of the non-metric cycles (we start by fixing triangles, then four-cycles and so on). This will make all non-metric cycles considered in each step of the construction induced. Because $\str{A}$ is a metric space and thus a complete graph, we will only need to consider
partial automorphisms of the non-metric cycles which move at most two vertices.
This makes it easy to fix every non-metric cycle by unwinding it to a ``Möbius strip'' as depicted at Figure~\ref{fig:mobius}.
\begin{figure}
\centering
\includegraphics{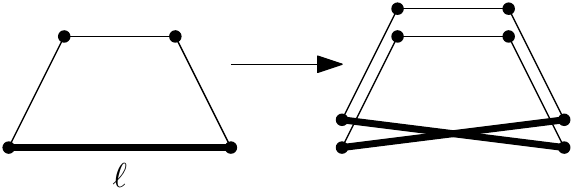}
\caption{Expansion of a non-metric cycle with longest edge $\ell$ to a ``Möbius strip''.}
\label{fig:mobius}
\end{figure}
\begin{proof}[Proof of Theorem~\ref{mainresult}]
Given a metric space $\str{A}$, let $N$ be an integer greater than the ratio of the largest distance in $\str{A}$ and the smallest distance in $\str{A}$.

Let $\str C_2$ be the $\posreals$-edge-labelled given by Proposition~\ref{prop:eppa1} applied on $\str A$ and let $\str A_2$ be the copy of $\str A$ in $\str C_2$. We then
build a sequence of $\posreals$-edge-labelled graphs $\str{C}_3,\str{C}_4,\ldots, \str{C}_N$
 such that for every $2\leq i\leq N$ the following conditions are satisfied:
\begin{enumerate}[label=(\Roman*)]
\item\label{prop1} $\str{C}_i$ contains an isomorphic copy $\str{A}_i$ of $\str{A}$ as a subgraph,
\item\label{prop2} every partial automorphism of $\str{A}_i$ extends to an automorphism of $\str{C}_i$, and,
\item\label{prop3} $\str{C}_i$ contains no non-metric cycles with at most $i$ vertices.
\end{enumerate}
First we show that from the existence of $\str{C}_N$ the theorem follows.
Observe that by the choice of $N$ every non-metric cycle has fewer than $N$
vertices and thus $\str{C}_N$ contains no non-metric cycles.  
Without loss of generality we can assume that $\str{C}_N$ is connected (otherwise we simply take 
the connected component of $\str{C}_N$ containing $\str{A}_N$) and thus we can apply Observation~\ref{obs:shortest}.
Let $\str{C}$ be the shortest path completion of $\str{C}_N$.
Because every automorphism of $\str{C}_N$ is also automorphism of
$\str{C}$ and $\str{A}_N$ is a subgraph of $\str{C}$ we get that $\str{C}$
extends all partial isometries of $\str{A}_N$ (which is isomorphic to
$\str{A}$).

\medskip

It remains to give the construction of $\str{C}_{i+1}$ from $\str{C}_i$ satisfying conditions \ref{prop1}--\ref{prop3}. 
A subset $M$ of $C_i$ is called \emph{bad} if $|M|=i+1$ and $\str{C}_i$ induces a non-metric cycle on $M$.
For $x\in C_i$ denote by $U(x)$ the family of all bad sets $M$ containing $x$.

We construct $\str{C}_{i+1}$ as follows:
\begin{itemize}
 \item Vertices of $\str{C}_{i+1}$ are pairs $(x,\chi_x)$ where $x\in C_i$ and $\chi_x$ is a function from $U(x)$ to $\{0,1\}$. We call such $\chi_x$ \emph{valuation function}.
 \item $(x,\chi_x)$ and $(y,\chi_y)$ are connected by an edge of length $\ell$ if and only if $d_{\str C_i}(x,y)=\ell$ and for every $M\in U(x)\cap U(y)$ one of the following holds:
\begin{enumerate}
 \item[(a)]\label{ec1} $x,y$ is the longest edge of the non-metric cycle induced on $M$ and $\chi_x(M)\neq \chi_y(M)$, or
 \item[(b)]\label{ec2} $x,y$ is not the longest edge of the non-metric cycle induced on $M$ and $\chi_x(M) = \chi_y(M)$.
\end{enumerate}
(These rules describe the ``Möbius strip'' of every bad set.)
\end{itemize}
\medskip
There are no other edges in $\str{C}_{i+1}$.
This finishes the construction of $\str{C}_{i+1}$. We now verify that $\str{C}_{i+1}$ satisfies conditions \ref{prop1}--\ref{prop3}.
\medskip

\noindent {\bf \ref{prop1}}:
We give an explicit description of an embedding $\psi$ of $\str{A}_i$ to $\str{C}_{i+1}$ and put $\str{A}_{i+1}$ to be the structure induced by $\str{C}_{i+1}$ on $\{\psi(x);x\in A_i\}$.

For every bad set $M\subseteq C_i$ such that $M\cap A_i\neq \emptyset$ we define a function $\chi_M\colon M\cap A_i\to \{0,1\}$. By definition, $M$ is bad because $\str C_i$ induces a non-metric cycle on $M$. Since $\str A$ is complete and it is a metric space (hence contains no non-metric triangles), it follows that $M\cap A$ consists either of one vertex or two vertices connected by an edge of the cycle.  Consider now two cases:
\begin{enumerate}
 \item $M\cap A=\{x,y\}$ where $\{x,y\}$ is the long edge of the non-metric cycle induced on $M$. In this case we put $\chi_M(x)=0$ and $\chi_M(y)=1$. (Notice that this step is not uniquely defined because the choice of $x$ and $y$ can be exchanged and it is indeed the purpose of the function $\chi_M$ to fix this choice.)
 \item $M$ does not intersect with $A$ by a long edge. In this case put $\chi_M(x)=0$ for all $x\in M\cap A$.
\end{enumerate}
Now we define a mapping $\psi$ from $A_i$ to $C_{i+1}$ by putting $\psi(x)=(x,\chi_x)$
where $\chi_x(M)=\chi_M(x)$ and put $A_{i+1}=\psi(A_i)$.  It is easy check that $\psi$ is an
embedding $\str{A}_i\to \str{C}_{i+1}$ because we chose functions $\chi_M$ in a way so that all edges
are preserved.  This verifies condition~\ref{prop1}.
\medskip

\noindent{\bf \ref{prop2}}:
We show that $\str{C}_{i+1}$ extends all partial automorphisms of $\str{A}_{i+1}$.

Consider any partial automorphism $\varphi$ of $\str{A}_{i+1}$.
Define $p\colon C_{i+1}\to C_i$ to be the \emph{projection} which maps every $(x,\chi_x)\in C_{i+1}$ to $x\in C_i$.
By $p$ we project the partial automorphism $\varphi$ of $\str{A}_{i+1}$ to a partial automorphism $p\circ \varphi\circ p^{-1}$
of $\str{A}_i$.
Denote by $\hat \varphi$ an extension of the partial automorphism $p\circ \varphi\circ p^{-1}$ of $\str A_i$ to
an automorphism of $\str{C}_i$ (which always exist by the induction hypothesis).

Let $F$ consist of all bad sets $M\subseteq C_i$ with the property that $M\cap A_i\neq \emptyset$ and there exists $x\in
M$, such that $(x,\chi_x)=\psi(x)\in \dom (\varphi)$ and $\chi_x(M)\neq \chi_y(\hat\varphi(M))$ where $(y,\chi_y)$ is such that $\varphi((x,\chi_x))=(y,\chi_y)$ (these are bad sets whose valuations are \emph{flipped} by $\varphi$).

We build an automorphism $\theta$ of $\str{C}_{i+1}$ by putting
 $\theta((x,\chi_x))=(\hat \varphi(x),\chi')$ where $\chi'(\hat\varphi(M))=\chi_x(M)$
if $M\notin F$ and $1-\chi_x(M)$ if $M\in F$. 
To verify that $\theta$ is indeed an automorphism first check that $\theta$ is one-to-one
because it is possible to construct its inverse.  Because the action of $\theta$ on the valuation functions does not affect the outcome of conditions for edges
in the construction of $\str{C}_{i+1}$, we get that $\theta$ is an isomorphism.

It remains to verify that $\theta$ extends $\varphi$.  This follows from the fact
that for every bad set $M$ it holds that $|M\cap \dom (\varphi)| \leq 2$.
Moreover, whenever $M\cap \dom (\varphi) = \{x,y\}$, $x\neq y$, $\varphi(x)=(x',\chi_{x'})$, $\varphi(y)=(y',\chi_{y'})$ then
$\chi_x(M) = \chi_{x'}(\hat\varphi (M))$ if and only if $\chi_y(M) = \chi_{y'}(\hat\varphi (M))$.
This finishes the proof of condition~\ref{prop2}.

\medskip

\noindent{\bf \ref{prop3}}:
Consider any set $M\subseteq \str{C}_{i+1}$ such that $|M|\leq i+1$ and the subgraph induced 
by $\str C_{i+1}$ on $M$ contains a non-metric cycle as a subgraph.  It follows that its
projection $p(M)$ contains a non-metric cycle in $\str{C}_i$.  By the induction
hypothesis we thus know that $|M|=i+1$ and $p(M)$ is a bad set (that is, $\str C_i$ induces a non-metric cycle on $p(M)$). Because of the projection of $\str C_{i+1}$ to $\str C$ it follows that $\str C_{i+1}$ induces a non-metric cycle on $M$.  Let
$(x,\chi_x),(y,\chi_y)$ be longest edge of this non-metric cycle. From the definition of the edges of $\str C_{i+1}$
we know that $\chi_x(M)\neq \chi_y(M)$.  Following the short edges of the cycle,
we however get $\chi_x(M)=\chi_y(M)$ a contradiction.
\end{proof}
\begin{remark*}
We in fact prove that the class of all finite metric spaces has the coherent extension property for partial isometries as defined by Solecki and Siniora~\cite{solecki2009,Siniora}: In Proposition~\ref{prop:eppa1} it is enough to fix a linear order on $U$ and extend the permutation in an order-preserving way. The coherency then goes through the proof of Theorem~\ref{mainresult}, it is enough to realise that ``flips compose''.
\end{remark*}
\begin{remark*}
This proof generalises to many known binary and general classes which are known to have
the extension property for partial automorphisms (see~\cite{Hubicka2017sauerconnant,Konecny2018b,Aranda2017} for examples of classes of structures having a variant of shortest path completion). This is going to appear in~\cite{eppatwographs,Hubicka2017sauer}.

There are classes, for which it is unknown whether they have EPPA or not.
Prominent among them are the class of all finite tournaments (see~\cite{Sabok} for partial results)
and the class of all finite partial Steiner triple systems~\cite{Hubicka2017designs}.
\end{remark*}
\section*{Acknowledgment}
We would like to thank the anonymous referee for remarks and corrections
that improved presentation of this paper and for the incredible speed in which they were delivered.

\bibliography{ramsey.bib}

\newcommand{\etalchar}[1]{$^{#1}$}
\begin{thebibliography}{ABWH{\etalchar{+}}17}

\bibitem[ABWH{\etalchar{+}}17]{Aranda2017}
Andres Aranda, David Bradley-Williams, Jan Hubi{\v c}ka, Miltiadis Karamanlis,
  Michael Kompatscher, Mat{\v e}j Kone{\v c}n{\'y}, and Micheal Pawliuk.
\newblock Ramsey expansions of metrically homogeneous graphs.
\newblock Submitted, arXiv:1707.02612, 2017.

\bibitem[EHKN18]{eppatwographs}
David Evans, Jan Hubi{\v{c}}ka, Mat{\v {e}}j Kone{\v {c}}n{\'{y}}, and Jaroslav
  Ne\v{s}et\v{r}il.
\newblock E{P}{P}{A} for two-graphs.
\newblock in preparation, 2018.

\bibitem[EHN17]{Evans3}
David~M. Evans, Jan Hubi{\v c}ka, and Jaroslav Ne{\v{s}}et{\v{r}}il.
\newblock {R}amsey properties and extending partial automorphisms for classes
  of finite structures.
\newblock 2017.

\bibitem[Hal49]{hall1949}
Marshall Hall.
\newblock Coset representations in free groups.
\newblock {\em Transactions of the American Mathematical Society},
  67(2):421--432, 1949.

\bibitem[HKN17]{Hubicka2017sauerconnant}
Jan Hubi{\v{c}}ka, Mat{\v{e}}j Kone{\v{c}}n{\'y}, and Jaroslav
  Ne{\v{s}}et{\v{r}}il.
\newblock Conant's generalised metric spaces are {R}amsey.
\newblock {\em arXiv:1710.04690, accepted to Contributions to Discrete
  Mathematics}, 2017.

\bibitem[HKN18]{Hubicka2017sauer}
Jan Hubi{\v{c}}ka, Mat{\v {e}}j Kone{\v {c}}n{\'{y}}, and Jaroslav
  Ne\v{s}et\v{r}il.
\newblock Semigroup-valued metric spaces: {R}amsey expansions and {E}{P}{P}{A}.
\newblock in preparation, 2018.

\bibitem[HL00]{herwig2000}
Bernhard Herwig and Daniel Lascar.
\newblock Extending partial automorphisms and the profinite topology on free
  groups.
\newblock {\em Transactions of the American Mathematical Society},
  352(5):1985--2021, 2000.

\bibitem[HN16]{Hubicka2016}
Jan Hubi{\v{c}}ka and Jaroslav Ne\v{s}et\v{r}il.
\newblock All those {R}amsey classes ({R}amsey classes with closures and
  forbidden homomorphisms).
\newblock Submitted, arXiv:1606.07979, 58 pages, 2016.

\bibitem[HN17]{Hubicka2017designs}
Jan Hubi{\v{c}}ka and Jaroslav Ne\v{s}et\v{r}il.
\newblock Ramsey theorem for designs.
\newblock {\em Electronic Notes in Discrete Mathematics}, 61:623 -- 629, 2017.
\newblock The European Conference on Combinatorics, Graph Theory and
  Applications (EUROCOMB'17).

\bibitem[HO03]{hodkinson2003}
Ian Hodkinson and Martin Otto.
\newblock Finite conformal hypergraph covers and {G}aifman cliques in finite
  structures.
\newblock {\em Bulletin of Symbolic Logic}, 9(03):387--405, 2003.

\bibitem[Hod02]{hodkinson}
Ian Hodkinson.
\newblock Finite model property for guarded fragments.
\newblock slides available at
  http://www.cllc.vuw.ac.nz/LandCtalks/imhslides.pdf, 2002.

\bibitem[HPSW18]{Sabok}
Jingyin Huang, Michael Pawliuk, Marcin Sabok, and Daniel Wise.
\newblock The {H}rushovski property for hypertournaments and profinite
  topologies.
\newblock preprint, 2018.

\bibitem[Kon18]{Konecny2018b}
Mat{\v e}j Kone{\v c}n{\'y}.
\newblock Semigroup-valued metric spaces.
\newblock Master thesis in preparation, Charles University, 2018.

\bibitem[Mac66]{mackey1966}
George~W Mackey.
\newblock Ergodic theory and virtual groups.
\newblock {\em Mathematische Annalen}, 166(3):187--207, 1966.

\bibitem[Ne{\v{s}}07]{Nevsetvril2007}
Jaroslav Ne{\v{s}}et{\v{r}}il.
\newblock Metric spaces are {R}amsey.
\newblock {\em European Journal of Combinatorics}, 28(1):457--468, 2007.

\bibitem[NR77]{Nevsetvril1977}
Jaroslav Ne{\v{s}}et{\v{r}}il and Vojt{\v{e}}ch R{\"o}dl.
\newblock A structural generalization of the {R}amsey theorem.
\newblock {\em Bulletin of the American Mathematical Society}, 83(1):127--128,
  1977.

\bibitem[Ott17]{otto2017}
Martin Otto.
\newblock Amalgamation and symmetry: From local to global consistency in the
  finite.
\newblock {\em arXiv:1709.00031}, 2017.

\bibitem[Pes08]{Pestov2008}
Vladimir~G Pestov.
\newblock A theorem of {H}rushovski--{S}olecki--{V}ershik applied to uniform
  and coarse embeddings of the {U}rysohn metric space.
\newblock {\em Topology and its Applications}, 155(14):1561--1575, 2008.

\bibitem[Ros11]{rosendal2011}
Christian Rosendal.
\newblock Finitely approximate groups and actions part {I}: The
  {R}ibes--{Z}alesski{\u\i} property.
\newblock {\em The Journal of Symbolic Logic}, 76(04):1297--1306, 2011.

\bibitem[RZ93]{Ribes1993}
Luis Ribes and Pavel~A Zalesskii.
\newblock On the profinite topology on a free group.
\newblock {\em Bulletin of the London Mathematical Society}, 25(1):37--43,
  1993.

\bibitem[Sab17]{sabok2017automatic}
Marcin Sabok.
\newblock Automatic continuity for isometry groups.
\newblock {\em Journal of the Institute of Mathematics of Jussieu}, pages
  1--30, 2017.

\bibitem[Sol05]{solecki2005}
S{\l}awomir Solecki.
\newblock Extending partial isometries.
\newblock {\em Israel Journal of Mathematics}, 150(1):315--331, 2005.

\bibitem[Sol09]{solecki2009}
S{\l}awomir Solecki.
\newblock Notes on a strengthening of the {H}erwig--{L}ascar extension theorem.
\newblock Unpublished note, 2009.

\bibitem[SS17]{Siniora}
Daoud Siniora and S{\l}awomir Solecki.
\newblock Coherent extension of partial automorphisms, free amalgamation, and
  automorphism groups.
\newblock {\em arXiv:1705.01888}, 2017.

\bibitem[Ver08]{vershik2008}
Anatoly~M. Vershik.
\newblock Globalization of the partial isometries of metric spaces and local
  approximation of the group of isometries of {U}rysohn space.
\newblock {\em Topology and its Applications}, 155(14):1618--1626, 2008.

\bibitem[Ver18]{Vershikprivate}
Anatoly~M. Vershik.
\newblock Personal communication.
\newblock July 28, 2018.

\end{thebibliography}
\end{document}